\documentclass[12pt,reqno]{amsart}
\usepackage{amsmath,amssymb,amsthm,mathtools}
\usepackage[hidelinks]{hyperref}
\usepackage{microtype}
\usepackage[margin=0.94in]{geometry}
\numberwithin{equation}{section}
\newcommand{\F}{\mathrm F}
\newcommand{\Fp}{\mathrm F^*}
\newcommand{\C}{\mathrm C}
\newcommand{\Op}{\mathrm O}
\newcommand{\Syl}{\operatorname{Syl}}
\newtheorem{theorem}{Theorem}[section]
\newtheorem{lemma}[theorem]{Lemma}
\newtheorem{proposition}[theorem]{Proposition}

\theoremstyle{remark}
\newtheorem{remark}[theorem]{Remark}

\title[Abelian quotients of subgroups]{Abelian quotients and orbit sizes for subgroups of linear groups}
\author{YONG YANG}
\address{Department of Mathematics, Texas State University, San Marcos, TX 78666, USA.}
\email{yang@txstate.edu}
\date{}

\begin{document}
\begin{abstract}
Keller and Yang formulated two conjectures relating abelian quotients of subgroups of finite linear groups to largest orbit sizes. After counterexamples were found to the original conjectures, two revised versions were proposed. We prove both revised conjectures for every finite group $G$ satisfying $\F(G)=\Fp(G)$. We also extend a regular-orbit theorem for nilpotent subgroups to completely reducible modules of mixed characteristic under the same hypothesis.
\end{abstract}
\maketitle

\section{Introduction}
For a finite group $H$ acting on a finite set $\Omega$, set
\[
 m_H(\Omega)=\max_{\omega\in\Omega}|H:\C_H(\omega)|.
\]
For a finite $H$-module $V$ we write simply $m_H(V)$. Keller and the author proved that if $H$ is solvable and $V$ is a faithful completely reducible $H$-module, then
\[
 |H/H'|\leq m_H(V).
\]
For a subgroup $U\leq G$, however, complete reducibility of $V$ as a $G$-module need not pass to $U$ in defining characteristic. Thus the subgroup problem does not follow by simply applying the theorem to $U$.

Keller and Yang raised two conjectures concerning abelian quotients of subgroups and largest orbit sizes; see \cite[Conjectures~2.7 and~2.8]{KellerYang2016}. In the notation used here, their first conjecture asserted that if $G$ is a finite group, $V$ is a finite faithful completely reducible $G$-module, possibly of mixed characteristic, $U\leq G$, and $M=m_G(V)$, then
\[
 |U/U'|\leq M.
\]
Their second conjecture asserted the same conclusion when $V$ is a finite-dimensional faithful $G$-module over a field of characteristic $p$ and $\Op_p(G)=1$, again with $M=m_G(V)$. In \cite{Yang2022}, the author showed that both conjectures are false in general. That paper then formulated two revised conjectures, strengthening the orbit term from the largest $G$-orbit to the largest $U$-orbit. Conjecture~3.1 assumes $\F(G)=\Fp(G)$ and that $V$ is a finite faithful completely reducible $G$-module, possibly of mixed characteristic; it predicts
\[
 |U/U'|\leq m_U(V)
\]
for every $U\leq G$. Conjecture~3.2 assumes $\F(G)=\Fp(G)$, that $V$ is a finite-dimensional faithful $G$-module over a field of characteristic $p$, and that $\Op_p(G)=1$, and predicts the same inequality. Moreover, \cite[Theorem~3.3]{Yang2022} shows that Conjecture~3.1 implies Conjecture~3.2.

Our main result proves the first revised conjecture.

\begin{theorem}\label{thm:main}
Let $G$ be a finite group with $\F(G)=\Fp(G)$, let $V$ be a finite faithful completely reducible $G$-module, possibly of mixed characteristic, and let $U\leq G$. Then
\[
 |U/U'|\leq m_U(V).
\]
\end{theorem}

Thus Conjecture~3.1 of \cite{Yang2022} holds. By \cite[Theorem~3.3]{Yang2022}, Conjecture~3.2 holds as well.

We first prove the solvable case. For solvable $G$ the condition $\F(G)=\Fp(G)$ is automatic.

\begin{theorem}\label{thm:solvable}
Let $G$ be a finite solvable group, let $V$ be a finite faithful completely reducible $G$-module, possibly of mixed characteristic, and let $U\leq G$. Then
\[
 |U/U'|\leq m_U(V).
\]
\end{theorem}

Taking $U=G$ in Theorem~\ref{thm:solvable} recovers the theorem of Keller and Yang \cite{KellerYang2016}. Thus Theorem~\ref{thm:solvable} extends their result from the acting solvable linear group itself to arbitrary subgroups. Theorem~\ref{thm:main} will follow from Theorem~\ref{thm:solvable} by the reduction in Section~5.

We prove Theorem~\ref{thm:solvable} by a minimal-counterexample argument, following the induction used in \cite{KellerYang2016}. The additional difficulty is that complete reducibility for $G$ need not imply complete reducibility for $U$. We therefore apply induction only to solvable quotient groups that act faithfully and completely reducibly on the modules under consideration.

To deduce Theorem~\ref{thm:main} from Theorem~\ref{thm:solvable}, we use an argument of Carlip for nilpotent subgroups \cite{Carlip1994}; compare \cite[Step~2 of the proof of Theorem~3.1]{Yang2011}. We prove the mixed-characteristic version needed here in Section~5.

In the reducible-module induction in \cite{Yang2022}, one passes to quotient actions attached to proper submodules. An inductive use of the conjectural hypothesis would require the resulting quotient group to satisfy $\F=\Fp$, a condition that is not preserved by arbitrary quotients. The proof here avoids that induction.

\section{A minimal counterexample}
Suppose that Theorem~\ref{thm:solvable} is false. Choose a counterexample with $|G|$ minimal, and subject to this, with $|U|$ minimal. Among these choices, take $V$ of minimal order. We begin with the induction needed below.

Suppose that $V=V_1\oplus V_2$ with nonzero $G$-submodules $V_1,V_2$, and put $N=\C_U(V_1)$. Let $m_1$ be the largest orbit size of $U$ on $V_1$, and let $m_2$ be the largest orbit size of $N$ on $V_2$. Since $N$ centralizes $V_1$ and $V$ is faithful, $N\cap\C_G(V_2)=1$, so $N$ acts faithfully on $V_2$.

The image of $G$ on $V_1$ is the solvable group $G/\C_G(V_1)$, and the image of $U$ is $U/N$. Likewise, the image of $G$ on $V_2$ is $G/\C_G(V_2)$, and the image of $N$ is isomorphic to $N$. Since $V_1$ and $V_2$ are direct summands of the completely reducible $G$-module $V$, both are completely reducible for the corresponding quotient groups. If either kernel is nontrivial, the corresponding quotient group is smaller than $G$; if the kernel is trivial, the group remains $G$ but the module has smaller order. Hence the minimal choice of $(G,U,V)$ gives
\begin{align*}
 m_1&\geq \left|U/N:(U/N)'\right|,\\
 m_2&\geq |N/N'|.
\end{align*}
Choose $v_i\in V_i$ with orbit size $m_i$, and put $v=v_1+v_2$. Then
\begin{align*}
 |v^U|
 &=|U:\C_U(v_1)\cap\C_U(v_2)|\\
 &=|U:\C_U(v_1)|\,
   |\C_U(v_1):\C_U(v_1)\cap\C_U(v_2)|\\
 &\geq m_1\,|N:N\cap\C_U(v_2)|\\
 &=m_1m_2,
\end{align*}
Since $N\lhd U$, we have
\[
 |U/U'|=|U:NU'|\,|N:N\cap U'|.
\]
Now $(U/N)'=U'N/N$, and $N'\leq N\cap U'$. Hence
\[
 |U/U'|\leq |U/N:(U/N)'|\,|N/N'|.
\]
Thus $m_U(V)\geq |U/U'|$, a contradiction. Therefore $V$ is irreducible, over a finite field of characteristic $p$.

Let $F=\F(G)$. Since $G$ is solvable, faithful and irreducible on $V$, one has $\Op_p(G)=1$, and hence $F$ is a $p'$-group. Also $\C_G(F)\leq F$. Thus $\Op_p(UF)=1$: if $P=\Op_p(UF)$, then $[P,F]\leq P\cap F=1$, so $P\leq\C_G(F)\leq F$, and therefore $P=1$.

Set $X=(UF)\ltimes V$. Write $\F(X)=P\times Q$, where $P$ is a $p$-group and $Q$ is a $p'$-group. Then $PV/V$ is a normal $p$-subgroup of $X/V\cong UF$, so $P\leq V$. Moreover $[Q,V]\leq Q\cap V=1$, and faithfulness of the $UF$-action gives $Q=1$. Hence
\begin{equation}\label{eq:FX}
 \F(X)=V.
\end{equation}
Thus $\Phi(X)\leq V$, and $\Phi(X)<V$, since $X=(UF)V=(UF)\Phi(X)$ would otherwise imply $X=UF$.

Put
\[
 \overline X=X/\Phi(X),\qquad
 \overline V=V/\Phi(X),\qquad
 \overline{UF}=UF\Phi(X)/\Phi(X).
\]
Because $X$ is solvable, $\Phi(\overline X)=1$ and \eqref{eq:FX} imply that $\overline V=\F(\overline X)$ is a direct product of minimal normal subgroups of $\overline X$. Hence $\overline V$ is a completely reducible $\overline{UF}$-module. We check faithfulness. Let $C=\C_{\overline{UF}}(\overline V)$. If $A$ is a $p'$-subgroup of $C$, then $[V,A]\leq\Phi(X)$. By coprime action,
\[
 V=[V,A]\C_V(A)=\Phi(X)\C_V(A).
\]
Hence $X=\Phi(X)(UF)\C_V(A)$, and the defining property of the Frattini subgroup gives $V=\C_V(A)$. Faithfulness on $V$ yields $A=1$. Thus $C$ is a $p$-group. Since $\Op_p(UF)=1$, we obtain $C=1$.

If $\Phi(X)>1$, then $|\overline V|<|V|$. Moreover $\overline{UF}\cong UF$ because $\Phi(X)\leq V$ and $UF\cap V=1$. If $UF<G$, the group has smaller order; if $UF=G$, the group $G$ and subgroup $U$ are unchanged while the module has smaller order. In either case minimality gives the desired inequality for the image of $U$ on $\overline V$. For a lift $v\in V$ of a vector $\bar v\in\overline V$, one has $\C_U(v)\leq\C_U(\bar v)$, so the same inequality holds on $V$, a contradiction. Hence $\Phi(X)=1$. Consequently $V$ is a faithful completely reducible $UF$-module. If $UF<G$, minimality in $|G|$ gives the theorem for $(UF,V,U)$, a contradiction. Therefore
\begin{equation}\label{eq:basicG}
 G=U\F(G).
\end{equation}

We next show that $U$ is nilpotent. If a maximal subgroup $L<U$ does not contain $U'$, then $U=LU'$ and
\[
 |U/U'|=|L:L\cap U'|\leq |L/L'|.
\]
Minimality in $|U|$ gives $|L/L'|\leq m_L(V)\leq m_U(V)$, a contradiction. Thus $U'\leq\Phi(U)$. Consequently every maximal subgroup of $U$ contains $U'$ and is therefore normal. A finite group in which every maximal subgroup is normal is nilpotent; hence $U$ is nilpotent.

Finally, $\Op_p(U)\neq1$. Otherwise nilpotence gives $p\nmid|U|$. By Maschke's theorem, $V$ is a faithful completely reducible $U$-module, and the Keller--Yang theorem \cite{KellerYang2016}, applied to $U$, gives $|U/U'|\leq m_U(V)$, again a contradiction.

We have proved that a minimal counterexample satisfies
\begin{equation}\label{eq:basic}
 G=U\F(G),\qquad \Op_p(G)=1,
 \qquad U\text{ is nilpotent},\qquad \Op_p(U)\neq1.
\end{equation}

Put
\[
 M=m_U(V).
\]
Since $(G,V,U)$ is a counterexample,
\begin{equation}\label{eq:bad}
 M<|U/U'|.
\end{equation}

\begin{lemma}\label{lem:strict}
If $H<U$, then
\[
 |H/H'|<|U/U'|.
\]
\end{lemma}
\begin{proof}
By minimality in $|U|$, Theorem~\ref{thm:solvable} holds for $H$ in the same group $G$. Hence
\[
 |H/H'|\leq m_H(V)\leq m_U(V)=M<|U/U'|.
\]
\end{proof}

The strict inequality in Lemma~\ref{lem:strict} implies the class-two property used below.

\begin{lemma}\label{lem:class2}
Let $P\in\Syl(U)$. Then
\begin{equation}\label{eq:abmax}
 |K/K'|<|P/P'|\qquad(K<P).
\end{equation}
Consequently
\[
 P'\leq Z(P)
\]
for every Sylow subgroup $P$ of $U$, and therefore
\begin{equation}\label{eq:class2}
 U'\leq Z(U).
\end{equation}
\end{lemma}
\begin{proof}
Write $U=P\times R$, where $R$ is the product of the Sylow subgroups of $U$ different from $P$. If $K<P$, then $KR<U$. Lemma~\ref{lem:strict} gives
\[
 |K/K'|\,|R/R'|<|P/P'|\,|R/R'|,
\]
and cancellation proves \eqref{eq:abmax}. Thompson proved that a finite $p$-group $P$ satisfying \eqref{eq:abmax} has $P'\leq Z(P)$; see \cite{Thompson1969}. Applying this result gives $P'\leq Z(P)$. Since $U$ is the direct product of its Sylow subgroups, \eqref{eq:class2} follows.
\end{proof}

Following standard terminology, a finite group $P$ is called \emph{ab-maximal} if $|P/P'|>|K/K'|$ for every proper subgroup $K<P$. Thus \eqref{eq:abmax} says that every Sylow subgroup of $U$ is ab-maximal.

\section{The non-quasi-primitive case}
Assume that the irreducible action of $G$ on $V$ is not quasi-primitive. We use the prime-index argument from the proof of \cite[Theorem~2.3]{KellerYang2016}. It follows from \cite[Corollary~0.3]{ManzWolf} and applies to an irreducible solvable linear group whose quotient by the Fitting subgroup is nilpotent. In our situation \eqref{eq:basic} gives
\[
 G/\F(G)\cong U/(U\cap\F(G)),
\]
which is nilpotent. Thus the hypotheses of that reduction are satisfied. Consequently there are a prime $r$ and a normal subgroup $D\lhd G$ with $|G:D|=r$ such that
\begin{equation}\label{eq:decomp}
 V_D=V_1\oplus\cdots\oplus V_r,
\end{equation}
where the $V_i$ are homogeneous $D$-components and $G/D$ permutes them cyclically. Each $V_i$ is therefore a direct sum of irreducible $D$-modules. Consequently
\[
 D/\C_D(V_i)
\]
is a solvable group of order smaller than $|G|$ acting faithfully and completely reducibly on $V_i$.

Put
\[
 T=U\cap D.
\]
If $U\leq D$, then $V_D$ is completely reducible by \eqref{eq:decomp}; after factoring out $\C_D(V)$, the resulting solvable linear group has order smaller than $|G|$ and contains the image of $U$. Minimality in $|G|$ then gives the desired inequality, a contradiction. Hence $UD=G$, and therefore
\begin{equation}\label{eq:UT}
 |U:T|=r.
\end{equation}
Choose $x\in U\setminus T$ in the Sylow $r$-subgroup of $U$. Since $U$ is nilpotent, $x$ centralizes every Sylow $q$-subgroup of $T$ for $q\neq r$.

\begin{lemma}\label{lem:T}
We have
\begin{equation}\label{eq:Tfacts}
 T\cap U'=T',\qquad [T,x]\leq T',
 \qquad |U/U'|=r|T/T'|.
\end{equation}
\end{lemma}
\begin{proof}
Since $U/T$ is cyclic,
\[
 |U/U'|=r|T:T\cap U'|.
\]
Also
\[
 |T/T'|=|T:T\cap U'|\,|T\cap U':T'|.
\]
Modulo $T'$, the subgroup $T\cap U'$ is generated by the commutators $[T,x]$. Because $x$ centralizes all Sylow $q$-subgroups of $T$ with $q\neq r$, the quotient $(T\cap U')/T'$ is an $r$-group. Since $T<U$, Lemma~\ref{lem:strict} yields
\[
 |T/T'|<|U/U'|=r|T:T\cap U'|.
\]
Hence $|T\cap U':T'|<r$. It follows that $T\cap U'=T'$, and the other two assertions follow immediately.
\end{proof}

For $1\leq i\leq r$, set
\[
 K_i=\C_T(V_i).
\]
Number the components so that $K_i^x=K_{i+1}$, with indices taken modulo $r$.

\begin{lemma}\label{lem:Ki}
For every $i$,
\begin{equation}\label{eq:Ki}
 K_iT'=K_1T',\qquad K_i\leq Z(T),
\end{equation}
and $K_i$ is an $r$-group.
\end{lemma}
\begin{proof}
By Lemma~\ref{lem:T}, conjugation by $x$ is trivial on $T/T'$, and hence $K_iT'$ is independent of $i$. Write this common subgroup as $L$. By Lemma~\ref{lem:class2}, $T'\leq Z(T)$. Since $K_i\lhd T$,
\[
 [L,T]=[K_i,T]\leq K_i
\]
for every $i$. Thus
\[
 [L,T]\leq\bigcap_{i=1}^rK_i=\C_T(V)=1,
\]
where the last equality follows from the faithfulness of the original $G$-action on $V$.
Hence $L\leq Z(T)$, proving the second assertion in \eqref{eq:Ki}.

Now let $q\neq r$. The $q$-part of $K_i$ lies in the unique Sylow $q$-subgroup $Q$ of the nilpotent group $U$. Since $x$ centralizes $Q$, that $q$-part is fixed by conjugation by $x$; but conjugation by $x$ sends $K_i$ successively through all the $K_j$. Hence the $q$-part lies in $\bigcap_jK_j=1$. Therefore $K_i$ is an $r$-group.
\end{proof}

Fix $K=K_1$. The image of $T$ in
\[
 D/\C_D(V_1)
\]
is $T/K$. Since this quotient has order smaller than $|G|$, minimality gives a vector $v\in V_1$ whose $(T/K)$-orbit has size at least
\[
 |(T/K)/(T/K)'|=|T:KT'|.
\]
Equivalently,
\begin{equation}\label{eq:v}
 |T:\C_T(v)|\geq |T:KT'|.
\end{equation}
On the other hand, $K$ acts faithfully on
\[
 W_1=V_2\oplus\cdots\oplus V_r,
\]
because
\[
 \C_K(W_1)=K\cap\bigcap_{j=2}^rK_j=\bigcap_{j=1}^rK_j=1.
\]
The module $W_1$ is completely reducible as a $D$-module, being a direct sum of the homogeneous components $V_2,\ldots,V_r$. Hence the faithful quotient $D/\C_D(W_1)$ has no nontrivial normal $p$-subgroup: such a subgroup would act trivially on every irreducible constituent of $W_1$ by the standard fact that a normal $p$-subgroup acts trivially on every irreducible module in characteristic $p$. The abelian regular-orbit theorem \cite[Theorem~4.3]{Yang2011} therefore applies to the image of $K$ inside $D/\C_D(W_1)$. Since $K$ acts faithfully and is abelian, it gives
\[
 |K|\leq m_K(W_1)\leq |K|.
\]
Thus $K$ has a regular orbit on $W_1$. Choose $w\in W_1$ with $\C_K(w)=1$.

Since $K\leq\C_T(v)$,
\begin{align}
 |(v+w)^T|
 &=|T:\C_T(v+w)|\notag\\
 &\geq |T:\C_T(v)|\,
    |K:K\cap\C_T(w)|\notag\\
 &\geq |T:KT'|\,|K|\notag\\
 &=|T/T'|\,|K\cap T'|.
 \label{eq:vw}
\end{align}
If $K\cap T'\neq1$, then $K$ being an $r$-group gives $|K\cap T'|\geq r$, and \eqref{eq:vw}, together with Lemma~\ref{lem:T}, contradicts \eqref{eq:bad}. Therefore
\begin{equation}\label{eq:Kcap}
 K\cap T'=1.
\end{equation}

\begin{lemma}\label{lem:intersections}
If $i\neq j$, then
\[
 K_i\cap K_j=1.
\]
\end{lemma}
\begin{proof}
It is enough to prove $K_1\cap K_j=1$ for $j\neq1$. Let $k\in K_1\cap K_j$. Since $K_1^{x^{j-1}}=K_j$, choose $a\in K_1$ such that $a^{x^{j-1}}=k$. Lemma~\ref{lem:T} implies $aT'=kT'$. By \eqref{eq:Kcap}, $a=k$. Thus $k^{x^{j-1}}=k$.

The group $U$ has class at most two, so
\[
 [k,x]^{j-1}=1.
\]
Also $x^r\in T$ and $K_1\leq Z(T)$, whence
\[
 1=[k,x^r]=[k,x]^r.
\]
Since $(j-1,r)=1$, we obtain $[k,x]=1$. Thus $k$ is fixed by $x$ and belongs to every $K_i$, so $k\in\C_T(V)=1$.
\end{proof}

Consequently $K$ acts faithfully on every individual $V_j$ with $j\neq1$. The quotient $D/\C_D(V_j)$ acts faithfully and completely reducibly on $V_j$, so its normal $p$-core is trivial. Thus \cite[Theorem~4.3]{Yang2011} applies and gives $|K|\leq m_K(V_j)\leq|K|$. Hence $K$ has a regular orbit on each such component.

\begin{proposition}\label{prop:qp}
In a minimal counterexample, the action of $G$ on $V$ is quasi-primitive.
\end{proposition}
\begin{proof}
Suppose first that $r\geq3$. If $K\neq1$, choose $v\in V_1$ satisfying \eqref{eq:v} and choose $w\in V_2$ in a regular $K$-orbit. The centralizer calculation used in \eqref{eq:vw}, together with \eqref{eq:Kcap}, gives
\[
 |(v+w)^T|\geq |T/T'|.
\]
The support of $v+w$ is contained in $\{V_1,V_2\}$ and contains $V_2$, because $w$ lies in a regular orbit of the nontrivial group $K$. Every nontrivial power $x^j$, $1\leq j\leq r-1$, shifts this support to a different subset; if $v=0$, the singleton support is shifted through all $r$ components. Since $T$ preserves every $V_i$, vectors in the same $T$-orbit have the same support. Hence the $r$ corresponding $T$-orbits are distinct. Hence
\[
 |(v+w)^U|=r|(v+w)^T|
 \geq r|T/T'|=|U/U'|,
\]
contrary to \eqref{eq:bad}. If $K=1$ and $T>1$, induction for the faithful completely reducible action of $T$ on $V_1$ gives a vector $v\in V_1$ with $|v^T|\geq|T/T'|>1$, so $v\neq0$, and the same support argument applies. If $T=1$, then $U$ has prime order and the result is already known.

It remains to consider $r=2$. Assume $K\neq1$. Choose $v\in V_1$ as in \eqref{eq:v} and choose $w\in V_2$ in a regular $K$-orbit. Put $z=v+w$. Then
\[
 |z^T|\geq |T/T'|.
\]
If $z^T$ is not $U$-invariant, its two $T$-orbit translates are distinct, so
\[
 |z^U|=2|z^T|\geq2|T/T'|=|U/U'|,
\]
a contradiction. Hence suppose that $z^T$ is $U$-invariant. Put
\[
 C=\C_T(v).
\]
There is $t\in T$ such that $z^x=z^t$. Since $x$ interchanges $V_1$ and $V_2$, comparison of the $V_1$-components gives
\[
 w^x=v^t.
\]
Now $w^x$ is regular for $K_2=K^x$. Because $K_2\leq Z(T)$, the size of a $K_2$-orbit is constant along a $T$-orbit. Thus $v$ is also regular for $K_2$, and consequently
\begin{equation}\label{eq:Ccap}
 C\cap K_2=1.
\end{equation}
Since $K_2=\C_T(V_2)$, \eqref{eq:Ccap} says that $C$ acts faithfully on $V_2$.

Apply induction to the image of $C$ in the smaller linear quotient $D/\C_D(V_2)$. There is $y\in V_2$ such that
\[
 |y^C|\geq |C/C'|.
\]
Because $C=\C_T(v)$,
\begin{align*}
 |(v+y)^T|
 &\geq |T:C|\,|C/C'|\\
 &=\frac{|T|}{|C'|}
 =|T/T'|\,|T':C'|.
\end{align*}
If $C'<T'$, then this is at least $2|T/T'|=|U/U'|$, a contradiction. Hence
\[
 C'=T'.
\]
In particular $T'=C'\leq C$, while $K\leq C$ because $K$ fixes all of $V_1$. Thus $KT'\leq C$. But \eqref{eq:v} gives
\[
 |T:C|=|v^T|\geq |T:KT'|,
\]
and so $C=KT'$. Lemma~\ref{lem:Ki} gives $K,T'\leq Z(T)$; hence $C$ is abelian. Therefore
\[
 1=C'=T'.
\]
Lemma~\ref{lem:T} now gives $[K,x]=1$, so $K^x=K$. This contradicts Lemma~\ref{lem:intersections}, since $K\neq1$.

Finally, if $r=2$ and $K=1$, then for $T>1$ induction on the faithful completely reducible action of $T$ on $V_1$ supplies a vector $v\in V_1$ with $|v^T|\geq|T/T'|>1$, so $v\neq0$. The element $x$ moves the support from $V_1$ to $V_2$, so
\[
 |v^U|=2|v^T|\geq|U/U'|.
\]
If $T=1$, then $U$ is cyclic. This completes the proof.
\end{proof}

\section{The quasi-primitive case}
By Proposition~\ref{prop:qp}, a minimal counterexample is quasi-primitive.

\begin{lemma}\label{lem:metacyclic}
Let $P$ be a metacyclic $p$-group with $P'\leq Z(P)$. If
\[
 |K/K'|<|P/P'|\qquad(K<P),
\]
then $P$ is abelian.
\end{lemma}
\begin{proof}
Assume that $P$ is nonabelian. Choose a cyclic normal subgroup
\[
 A=\langle a\rangle\lhd P
\]
such that $P/A$ is cyclic, and choose $b\in P$ with $P/A=\langle bA\rangle$. Since $P$ has class at most two,
\[
 P'=\langle[a,b]\rangle
\]
is a nontrivial cyclic $p$-group. Put
\[
 L=\langle A,b^p\rangle.
\]
Since $P/A$ is a nontrivial cyclic $p$-group generated by $bA$, the subgroup $L/A=\langle b^pA\rangle$ has index $p$ in $P/A$. Thus $L<P$ and $|P:L|=p$. As $A$ is cyclic and $P'$ is central,
\[
 P'=\langle[a,b]\rangle,
 \qquad
 L'=\langle[a,b^p]\rangle
   =\langle[a,b]^p\rangle=(P')^p.
\]
Since $P'\neq1$ is a finite cyclic $p$-group,
\[
 |P':(P')^p|=p.
\]
Therefore
\[
 |L/L'|
 =\frac{|P|/p}{|P'|/p}
 =|P/P'|,
\]
contrary to the strict inequality for proper subgroups.
\end{proof}

If the quasi-primitive group $G$ is metacyclic, then every subgroup of $G$ is metacyclic: if $A\lhd G$ is cyclic with $G/A$ cyclic and $H\leq G$, then $H\cap A$ is cyclic and $H/(H\cap A)$ embeds in $G/A$. Thus every Sylow subgroup of our nilpotent $U$ is metacyclic. Lemmas~\ref{lem:class2} and \ref{lem:metacyclic} show that every Sylow subgroup of $U$ is abelian, and hence $U$ itself is abelian. The known abelian case \cite[Theorem~4.3]{Yang2011} gives a regular $U$-orbit on $V$, contradicting \eqref{eq:bad}.

We may therefore assume that $G$ is not metacyclic. By Berger's theorem, for finite solvable irreducible linear groups quasi-primitivity is equivalent to primitivity \cite{Berger1975}. Holt and Yang proved that a solvable primitive linear group which is not metacyclic either has a regular orbit or is one of the exceptional pairs in their Table~4.1 \cite[Theorem~1.1]{HoltYang2023}. If $G$ has a regular orbit, the same vector is regular for $U$, and again \eqref{eq:bad} is impossible.

It remains to deal with the exceptional cases in \cite[Table~4.1]{HoltYang2023}.

\begin{lemma}\label{lem:gap}
For every exceptional pair $(X,W)$ listed in \cite[Table~4.1]{HoltYang2023} and every nilpotent subgroup $H\leq X$,
\begin{equation}\label{eq:gapclaim}
 |H/H'|\leq m_H(W).
\end{equation}
\end{lemma}
\begin{proof}
Using GAP \cite{GAP}, the author checked the conjugacy classes of subgroups in each exceptional action in \cite[Table~4.1]{HoltYang2023}. For every nilpotent subgroup $H$, the computation verifies \eqref{eq:gapclaim} by comparing $|H/H'|$ with the largest $H$-orbit on $W$.
\end{proof}

Applying Lemma~\ref{lem:gap} to the nilpotent subgroup $U$ contradicts \eqref{eq:bad}. Hence a minimal counterexample does not exist, proving Theorem~\ref{thm:solvable}.

\section{Proof of the main theorem}
We first recall a standard fact about modular representations.

\begin{lemma}\label{lem:normalp}
Let $N\lhd H$ be a normal $r$-subgroup and let $S$ be an irreducible $H$-module in characteristic $r$. Then $N$ acts trivially on $S$.
\end{lemma}
\begin{proof}
A finite $r$-group acting on a nonzero vector space in characteristic $r$ has a nonzero fixed vector. Thus $\C_S(N)\neq0$. Since $N\lhd H$, the fixed-point space $\C_S(N)$ is $H$-invariant. Irreducibility gives $\C_S(N)=S$.
\end{proof}

We use the following elementary consequence of Clifford theory.

\begin{lemma}\label{lem:homogeneous-detection}
Let $X$ be a finite group, let $N\lhd X$ be an $r$-group, and let $T$ be an $X$-module over a field of characteristic different from $r$. Let $A$ be a homogeneous component of the completely reducible $N$-module $T_N$, and let
\[
 Y=\sum_{g\in X}A^g.
\]
If $Q\leq N$ acts nontrivially on $A$, then $Y$ contains an irreducible $X$-submodule on which $Q$ acts nontrivially.
\end{lemma}
\begin{proof}
The distinct conjugates of $A$ are distinct homogeneous $N$-components, so their sum is direct and $X$ permutes them transitively. Choose a minimal nonzero $X$-submodule $S\leq Y$. Since the characteristic is coprime to $|N|$, $S_N$ is completely reducible. Because $S_N$ is semisimple, it is the direct sum of its isotypic components. Each such isotypic component lies in the corresponding homogeneous component of $Y$, and hence equals the intersection of $S$ with that component. Therefore the set of components $B$ for which $S\cap B\neq0$ is a nonempty $X$-invariant subset of the transitive $X$-orbit of $A$; it is the full orbit. In particular $S\cap A\neq0$.

Write $A$ as a direct sum of copies of an irreducible $N$-module $L$. Since $Q$ acts nontrivially on $A$, one has $Q\not\leq\ker_N(L)$. Every irreducible constituent of the nonzero $N$-module $S\cap A$ is isomorphic to $L$. Thus $Q$ acts nontrivially on $S\cap A$, and hence on $S$.
\end{proof}

Carlip's argument \cite{Carlip1994} also works in mixed characteristic in the following form; compare \cite[Step~2 of the proof of Theorem~3.1]{Yang2011}.

\begin{lemma}\label{lem:detection}
Let $G$ be a finite group with $\F(G)=\Fp(G)$, let $V$ be a finite faithful completely reducible $G$-module, possibly of mixed characteristic, and let $U\leq G$. Put $F=\F(G)$ and $X=FU$. Then $V$ contains a completely reducible $X$-submodule $W$ such that
\[
 \C_U(W)=1.
\]
\end{lemma}
\begin{proof}
Let $1\neq x\in U$ have prime order. We show that some irreducible $X$-submodule $S_x$ of $V$ is not centralized by $x$.

Suppose first that $x\in F$, and let $|x|=r$. Since $F$ is nilpotent, $x\in \Op_r(F)=\Op_r(G)$. Because $V$ is faithful and completely reducible as a $G$-module, some irreducible $G$-summand $T$ is not centralized by $x$. By Lemma~\ref{lem:normalp}, the characteristic of $T$ is not $r$. Set $N=\Op_r(F)$ and $Q=\langle x\rangle$.

Now suppose that $x\notin F$. The generalized Fitting subgroup satisfies $\C_G(\Fp(G))\leq\Fp(G)$. Since $F=\Fp(G)$,
\[
 \C_G(F)\leq F,
\]
and therefore $[x,F]\neq1$. As $F$ is the direct product of its normal Sylow subgroups, there is a prime $r$ such that
\[
 Q=[x,\Op_r(F)]\neq1.
\]
Set $N=\Op_r(F)$. Choose $1\neq q\in Q$. Faithfulness of $V$ gives an irreducible $G$-summand $T$ on which $q$, and hence $Q$, acts nontrivially. Again Lemma~\ref{lem:normalp}, applied to $\Op_r(G)$, shows that the characteristic of $T$ is not $r$.

Thus in either case $N\lhd X$ is an $r$-group, $T$ is an irreducible $G$-summand of characteristic different from $r$, and $1\neq Q\leq N$ acts nontrivially on $T$. By Maschke's theorem, $T_N$ is completely reducible. Choose an $N$-homogeneous component $A$ of $T_N$ on which $Q$ acts nontrivially, and put $Y=\sum_{g\in X}A^g$. Lemma~\ref{lem:homogeneous-detection} gives an irreducible $X$-submodule $S\leq Y$ on which $Q$ acts nontrivially.

In the first case $Q=\langle x\rangle$, so $x$ acts nontrivially on $S$. In the second case $Q=[x,N]$. If $x$ centralized $S$, then $N$ stabilizes $S$ and every commutator in $[x,N]$ would act trivially on $S$, contrary to the choice of $S$, on which $Q$ acts nontrivially. Thus $x$ also acts nontrivially on $S$ in this case. Set $S_x=S$.

Finally let
\[
 W=\sum_{\substack{x\in U\\ |x|\text{ prime}}}S_x.
\]
The sum of simple $X$-submodules is contained in the $X$-socle of $V$ and is therefore completely reducible. If $\C_U(W)>1$, then Cauchy's theorem gives a nonidentity element $x\in\C_U(W)$ of prime order. This element centralizes $S_x$, contradicting its construction. Hence $\C_U(W)=1$.
\end{proof}

\begin{remark}\label{rem:detection-nilpotence}
Nilpotence of $U$ is not used in Lemma~\ref{lem:detection}; it is needed in the proof of Theorem~\ref{thm:main} only to ensure that $\F(G)U$ is solvable. The lemma uses the self-centralizing property of $\Fp(G)$ and coprime semisimplicity for one normal Sylow subgroup of $\F(G)$ at a time.
\end{remark}

\begin{proof}[Proof of Theorem~\ref{thm:main}]
Suppose the theorem is false, and choose a counterexample $(G,V,U)$ with $|U|$ minimal. Let $F=\F(G)=\Fp(G)$.

We first show that $U$ is nilpotent. If a maximal subgroup $L<U$ does not contain $U'$, then $U=LU'$ and
\[
 |U/U'|=|L:L\cap U'|\leq |L/L'|.
\]
By minimality, applied to the proper subgroup $L$ in the same group $G$ and on the same module $V$,
\[
 |L/L'|\leq m_L(V)\leq m_U(V),
\]
a contradiction. Thus $U'\leq\Phi(U)$. Hence $U/\Phi(U)$ is abelian, and therefore nilpotent. Since a finite group is nilpotent if and only if its Frattini quotient is nilpotent, $U$ is nilpotent.

Put $X=FU$. Then
\[
 X/F\cong U/(U\cap F)
\]
is nilpotent. Since $F$ is nilpotent, $X$ is solvable. Lemma~\ref{lem:detection} gives a completely reducible $X$-submodule $W\leq V$ such that $\C_U(W)=1$.

Let $C=\C_X(W)$, and write $\overline X=X/C$ and $\overline U=UC/C$. Then $\overline X$ is solvable and acts faithfully and completely reducibly on $W$. Moreover $U\cap C=1$, so $U\cong\overline U$ and $U/U'\cong\overline U/\overline U'$. The $U$-orbits and the $\overline U$-orbits on $W$ have the same sizes. Applying Theorem~\ref{thm:solvable} to $(\overline X,W,\overline U)$ gives
\[
 |U/U'|=|\overline U/\overline U'|\leq m_{\overline U}(W)=m_U(W)\leq m_U(V),
\]
contrary to the choice of $(G,V,U)$. This proves the theorem.
\end{proof}

\begin{remark}\label{rem:2022}
No induction on proper submodules is needed. This is important because the condition $\F=\Fp$ need not pass to arbitrary quotients. Lemma~\ref{lem:detection} instead works inside the original module, and Theorem~\ref{thm:solvable} is applied only to a quotient of the solvable group $\F(G)U$.
\end{remark}

The same reduction also gives a mixed-characteristic extension of the main regular-orbit theorem of \cite{Yang2011}.

\begin{theorem}\label{thm:mixed-regular}
Let $G$ be a finite group with $\F(G)=\Fp(G)$, let $V$ be a finite faithful completely reducible $G$-module, possibly of mixed characteristic, and let $H\leq G$ be nilpotent. Assume that $H$ involves no wreath product $C_2\wr C_2$. If $V$ has a nonzero component of characteristic $2$, assume in addition that $H$ involves no wreath product $C_r\wr C_r$ for any Mersenne prime $r$. Then $H$ has a regular orbit on $V$.
\end{theorem}
\begin{proof}
Put $F=\F(G)$ and $X=FH$. Since $X/F\cong H/(H\cap F)$ is nilpotent and $F$ is nilpotent, $X$ is solvable. By Lemma~\ref{lem:detection}, there is a completely reducible $X$-submodule $W\leq V$ such that $\C_H(W)=1$.

Decompose $W$ according to the characteristics of its irreducible $X$-summands:
\[
 W=\bigoplus_{p\in\pi} W_p,
\]
where $W_p$ is the sum of the summands in characteristic $p$, and zero terms are omitted. For $p\in\pi$, let $C_p=\C_X(W_p)$, set $X_p=X/C_p$, and let $H_p=HC_p/C_p$. Then $X_p$ is solvable and acts faithfully and completely reducibly on $W_p$. Hence
\[
 \Op_p(X_p)=1
\]
by Lemma~\ref{lem:normalp}. The group $H_p$ is nilpotent. In the terminology of \cite{Yang2011}, to ``involve'' a group means to have a section isomorphic to it. Hence a quotient of $H$ cannot involve one of the excluded wreath products unless $H$ already does. The hypotheses of \cite[Theorem~3.1]{Yang2011} therefore hold for $(X_p,W_p,H_p)$. Choose $w_p\in W_p$ in a regular $H_p$-orbit. Equivalently,
\[
 \C_H(w_p)=H\cap C_p=\C_H(W_p).
\]

Set $w=\sum_{p\in\pi}w_p\in W$. Since $W$ is the direct sum of its characteristic components,
\[
 \C_H(w)=\bigcap_{p\in\pi}\C_H(w_p)
 =\bigcap_{p\in\pi}\C_H(W_p)
 =\C_H(W)=1.
\]
Thus $w$ lies in a regular $H$-orbit on $V$.
\end{proof}



\section*{Acknowledgements}
This work was partially supported by a grant from the Simons Foundation (\#918096, to YY).

\section*{Disclosure Statement}
The author declares that he has no competing interests and no conflicts of interest.

\section*{Data Availability Statement}
Data sharing is not applicable to this article, as no data sets were generated or analysed during the current study.

\end{document}